\documentclass[12pt,a4paper]{article}
\usepackage{latexsym,hyperref,amssymb,amsfonts,amsmath,amsthm,nccmath,bm}
\usepackage{amsmath,amsthm,fullpage}
\usepackage{mdwlist}
\usepackage{pdfsync}
\usepackage{amssymb, amsfonts, graphicx}
\usepackage[usenames,dvipsnames]{xcolor}
\usepackage{enumitem}
\usepackage{multicol}
\usepackage[linesnumbered,ruled]{algorithm2e}

\usepackage[margin=1.5cm]{geometry}

\numberwithin{equation}{section}
\newtheorem{theorem}{Theorem}[section]

\newtheorem{lemma}[theorem]{Lemma}

\theoremstyle{definition}
\newtheorem{definition}[theorem]{Definition}

\newtheorem{example}[theorem]{Example}

\def \deg {{\rm deg}}

\def \leq{\leqslant}
\def \geq{\geqslant}
\def \preceq{\preccurlyeq}
\def \N{\mathbb{N}}
\def \O{\Omega}

\def \mod#1{{\:({\rm mod}\ #1)}}

\let\oldproofname=\proofname
\renewcommand{\proofname}{\rm\bf{\oldproofname}}

\newcommand{\ignore}[1]{}

\title{\bf MAX for $k$-independence in multigraphs}
\author{Nevena Franceti\'{c} \footnote{
School of Mathematical Sciences, Monash University, Victoria 3800, Australia} \qquad Sara Herke \footnote{School of Mathematics and Physics, The University of Queensland, QLD 4072, Australia} \qquad Daniel Horsley \footnotemark[1] }

\date{}

\begin{document}
\def\baselinestretch{1.2}\small\normalsize

\maketitle

\begin{abstract}
For a fixed positive integer $k$, a set $S$ of vertices of a graph or multigraph is called a {\it $k$-independent set} if the subgraph induced by $S$ has maximum degree less than $k$.  The well-known algorithm MAX finds a maximal $k$-independent set in a graph or multigraph by iteratively removing vertices of maximum degree until what remains has maximum degree less than $k$.  We give an efficient procedure that determines, for a given degree sequence $D$, the smallest cardinality $b(D)$ of a $k$-independent set that can result from any application of MAX to any loopless multigraph with degree sequence $D$.  This analysis of the worst case is sharp for each degree sequence $D$ in that there exists a multigraph $G$ with degree sequence $D$ such that some application of MAX to $G$ will result in a $k$-independent set of cardinality exactly $b(D)$.
\end{abstract}

\section{Introduction}

Unless otherwise specified, all sets in this paper are multisets and all multigraphs are loopless. Let $\N$ denote the set of nonnegative integers.  The {\it degree} of a vertex $v$ in a multigraph $G$, denoted $\deg_G(v)$, is the number of edges incident with $v$ in $G$. For a multigraph $G$ and a vertex $v$ of $G$, we define $G-v$ to be the multigraph obtained from $G$ by deleting $v$ and all of the edges incident with $v$. We use $\Delta(G)$ to denote the maximum degree of a multigraph $G$, and $\max(D)$ to denote the maximum element of a finite multiset $D$ of integers.

For our purposes, we define the \emph{degree sequence $D$ of a multigraph $G$} to be the multiset $\{\deg_G(v):v \in V(G)\}$.  We say a degree sequence with $n$ elements has {\it order} $n$. For conciseness, we use $\sum D = \sum_{z \in D} z$ for a finite multiset $D$ of integers.  It is well known that a finite multiset $D$ of nonnegative integers is the degree sequence of some multigraph if and only if $\sum D$ is even and $\sum D \geq 2\max(D)$ \cite{Ha}. When we say that a multiset is a \emph{degree sequence} we mean that it is the degree sequence of some multigraph.

In the first four sections of this paper we take $k$ to be a fixed positive integer. Many of the concepts and operations we define are implicitly dependent on $k$. A subset $S$ of the vertex set of a multigraph $G$ is
said to be {\it $k$-independent} if the subgraph of $G$ induced by $S$ has maximum
degree less than $k$. Setting $k = 1$ recovers the usual notion of an independent set. Generally one is interested in finding $k$-independent sets of large cardinality.  For a multigraph $G$, the {\it $k$-independence number} of $G$, denoted $\alpha_k(G)$, is the maximum cardinality of a $k$-independent set in $G$. Determining $\alpha_k(G)$ for an arbitrary graph $G$ is NP-Complete \cite{JaPe}. A survey on $k$-independence can be found in \cite{CFHV}.

One of the simplest and most studied algorithms for finding a $k$-independent set in a multigraph is the so-called MAX algorithm. The algorithm was initially introduced for finding a $1$-independent set in 1983 by Griggs \cite{Gr}, but it easily generalises to higher values of $k$. MAX operates by iteratively removing a vertex of maximum degree until the multigraph that remains has maximum degree less than $k$, see Algorithm \ref{alg:MAX}.

\begin{algorithm} \label{alg:MAX}
    \SetKwInOut{Input}{Input}
    \SetKwInOut{Output}{Output}
    \Input{$G$, a multigraph of order $n$}
    \Output{a maximal $k$-independent set of $G$}
    $H := G$\;
    \While{$\Delta(H) \geq k$}
      {
        Choose a vertex $v \in V(H)$ such that $\deg_{H}(v) = \Delta(H)$\;
        $H:=H-v$;
      }
	return $V(H)$\;
    \caption{MAX algorithm for finding a $k$-independent set}
\end{algorithm}

Note that the choice of a maximum degree vertex in line 3 of Algorithm \ref{alg:MAX} is arbitrary.   We use the phrase {\it for any application of MAX} to mean for any choice of maximum degree vertices throughout Algorithm \ref{alg:MAX} and the phrase {\it for some application of MAX} to mean for some choice of maximum degree vertices throughout Algorithm \ref{alg:MAX}.

The MAX algorithm and its variants have been extensively studied in the context of investigating $k$-independence numbers of graphs or multigraphs with given degree sequences \cite{BGHR,CaTu,Je,LBS1,LBS2,STY}. Most notably, Caro and Tuza's 1991 analysis of the MAX algorithm \cite{CaTu} yielded a closed form lower bound on $k$-independence number of a multigraph in terms of its degree sequence which remains the best general such bound known. As well as closed form lower bounds, analyses of MAX that give rise to procedural bounds have also been a topic of interest \cite{AmDaPe,FaMaSa,GrKl,Je,Se,Tr}. In particular, for simple graphs, Jelen \cite{Je} gives a procedural lower bound on the size of a $k$-independent set yielded by MAX based on the concept of ``$k$-residues''. He shows his bound is the best one that can be obtained as a weighted sum of the terms of the $k$-residue, but there are many degree sequences for which it is not tight. Jelen's work has been extended in \cite{AmDaPe}. Here we show that, for multigraphs, there is an alternative procedural bound that is tight for every degree sequence.

We next state the main result of this paper, noting that $b_k(D)$ is defined in Section~\ref{s:procedure} and is a positive integer that can be calculated in $O(\sum D)$ time for any degree sequence $D$. For each degree sequence $D$, Theorem \ref{kIndepBound} gives a tight procedural bound on the worst case behaviour of MAX on a multigraph with degree sequence $D$.

\begin{theorem}\label{kIndepBound}
Let $D$ be a degree sequence and $k$ be a fixed positive integer. Then any application of MAX to any multigraph with degree sequence $D$ will result in a $k$-independent set of cardinality at least $b_k(D)$.  Furthermore, there exists a multigraph $G$ with degree sequence $D$ such that some application of MAX to $G$ will result in a $k$-independent set of cardinality exactly $b_k(D)$.
\end{theorem}

In the case of lower bounds for the $k$-independence number of simple graphs in terms of their average degree, Caro and Tuza's result in \cite{CaTu} has since been improved upon by results based on more complicated procedures than MAX \cite{CaHa,Ko}. It is worth noting, however, that these methods do not appear to generalise readily to multigraphs.

In Section~\ref{s:procedure} we introduce some further notation and use this notation to describe the procedure for calculating $b_k(D)$.  In Section~\ref{s:mainproof} we prove our main result using a technical lemma (Lemma \ref{lem:main}) whose proof is deferred to Section~\ref{s:lemmaproofs}.  In Section~\ref{s:coverings} we discuss an application to finding lower bounds for the size of coverings. In Section~\ref{s:loops} we consider the variant of the problem in which we allow our multigraphs to have loops. We conclude with a short discussion of the case of simple graphs in Section~\ref{s:simple}.

\section{Definitions and preliminary results \label{s:procedure}}

For multisets of nonnegative integers $D$ and $E$, we define $D \uplus E$ and $D \setminus E$ and so that $\mu_{D \uplus E}(z)=\mu_{D}(z)+\mu_{E}(z)$ and $\mu_{D \setminus E}(z)=\max(0,\mu_{D}(z)-\mu_{E}(z))$, where $\mu_X(z)$ denotes the number of elements of the multiset $X$ equal to the integer $z$.

\begin{definition}[{\bf reduction}] Let $D$ be a degree sequence.  We say that a degree sequence $D'$ is a {\it reduction} of $D$ if there is a multigraph $G$ with degree sequence $D$ and a vertex $v \in V(G)$ with $\deg_G(v) = \Delta(G)$ such that $G - v$ has degree sequence $D'$.
\end{definition}

Note that if $D$ is a degree sequence of order $n$ and $D'$ is a reduction of $D$ then $D'$ is of order $n-1$ and $\sum D' = \sum(D)-2\max(D)$.

If $G$ is a multigraph of order $n$ with maximum degree less than $k$, then applying MAX to $G$ will trivially result in a $k$-independent set of cardinality $n$. Accordingly, we make the following definition.

\begin{definition}[{\bf trivial}]We say a degree sequence is \emph{trivial} if its maximum element is less than $k$. Otherwise it is \emph{nontrivial}.
\end{definition}

Given a degree sequence $D$ of order $n$, we now define a procedure to obtain another degree sequence $\Omega(D)$ of order $n-1$. If every reduction of $D$ is trivial, then we will set $\Omega(D)=\{0,\ldots,0\}$ to indicate this fact. Otherwise we will define $\Omega(D)$ to be a particular nontrivial reduction of $D$. We define $\Omega$ in Definition~\ref{def:Omega} and establish these properties in Lemma \ref{lm:properties of Omega(D)}.  It will turn out that one worst case of applying the MAX algorithm to a multigraph with degree sequence $D$ will produce multigraphs whose degree sequences are obtained by iteratively applying $\Omega$ to $D$.

\begin{definition}[{\bf $\bm{x}$-decrement}]
Let $E$ be a multiset of nonnegative integers and let $x$ be a positive element of $E$. We say that $D$ is obtained from $E$ by an {\it $x$-decrement} if $D=(E \setminus \{x\}) \uplus \{x-1\}$.
\end{definition}

\begin{definition}[{\bf $\bm{\Omega(D)}$, decrement sequence}]\label{def:Omega}
Let $D$ be a degree sequence of order $n$ and let $A_0=D \setminus \{\max(D)\}$. If $\sum A_0 < \max(D)+2k$ or if
$\max(A_0)<k$, then define $\Omega(D)$ to be the degree sequence $\{0,\ldots,0\}$ of order $n-1$. Otherwise, let $s = \sum A_0$ and let $A_1,\ldots,A_s$ be the sequence of multisets such that, for $i=1,\ldots,s$, $A_{i}$ is obtained from $A_{i-1}$ by an $a_i$-decrement, where
\begin{itemize}[nosep]
    \item
$a_i=\max(A_{i-1})$ if $\max(A_{i-1})>k$;
    \item
$a_i$ is the smallest positive element of $A_{i-1}$ otherwise.
\end{itemize}
We define the {\it decrement sequence} of $D$ to be $(a_1,\ldots,a_s)$ and we define $\Omega(D)$ to be $A_{\max(D)}$. (When $i \not\equiv \max(D) \mod{2}$ or $i >s-2k$, $A_i$ is not a degree sequence but we shall prove in Lemma~\ref{lm:properties of Omega(D)} that $\Omega(D)$ is always a degree sequence.)
\end{definition}

It is often useful to view a degree sequence $D$ as an integer partition of $\sum D$ (where we allow parts equal to 0) and to visualise $D$ in a Ferrers diagram where the elements of $D$ are given by the row lengths.

\begin{example}\label{ex:Omega}
If $k=3$ and $D=\{1,2,2,4,4,5,6\}$ then $s = 18$, the decrement
sequence of $D$ is $(5,4,4,4,1,2,1,2,1,3,2,1,3,2,1,3,2,1)$, and $\Omega(D) = \{0,1,2,3,3,3\}$. Figure~\ref{fig:Omega} shows Ferrers diagrams for $D$ and $\Omega(D)$ with the dashed lines indicating that $k=3$.
\end{example}

\begin{figure}[h!]
\centering
 \includegraphics[scale=0.5]{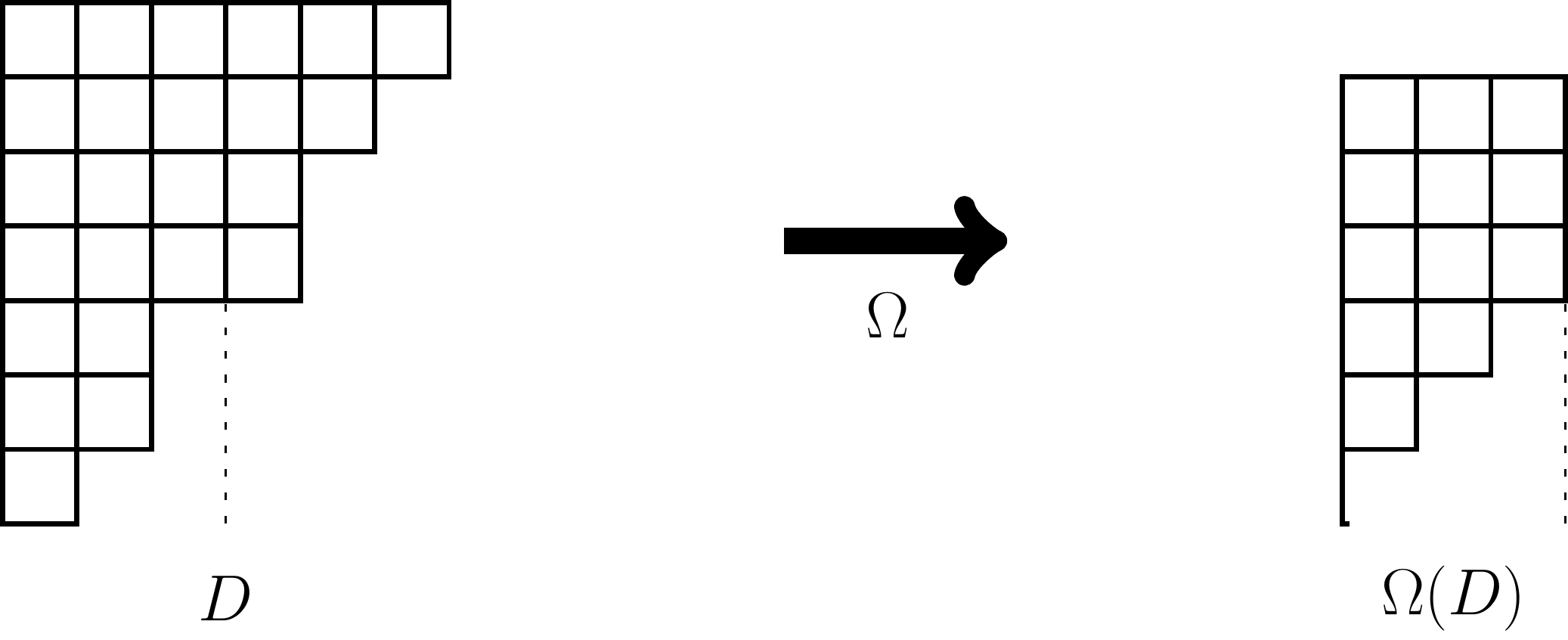}
\caption{Ferrers diagrams for $D$ and $\Omega(D)$ from Example~\ref{ex:Omega}}
\label{fig:Omega}
\end{figure}

Let $D$ be a degree sequence of order $n$. We let $\Omega^0(D)=D$, $\Omega^1(D)=\Omega(D)$, $\Omega^2(D)=\Omega(\Omega(D))$ and so on up to $\Omega^n(D)$. Observe that $\Omega^n(D)$ must be the empty set and hence $\Omega^n(D)$ is trivial.

\begin{definition}[$\bm{b(D)}$] Let $D$ be a degree sequence of order $n$. We define $b(D) = |\Omega^p(D)| = n-p$, where $p$ is the least nonnegative integer such that $\Omega^{p}(D)$ is trivial. We use the notation $b_k(D)$ when we wish to specify the value of $k$, or simply to emphasise the dependence on $k$.
\end{definition}

It is clear that, for any degree sequence $D$, $b(D)$ can be computed in $O(\sum D)$ time.

\begin{example}
Revisiting Example~\ref{ex:Omega}, let $k=3$ and $D=\{1,2,2,4,4,5,6\}$. Then $\Omega(D) = \{0,1,2,3,3,3\}$, $\Omega^2(D) = \{0,0,0,3,3\}$ and $\Omega^3(D) = \{0,0,0,0\}$.  Hence $b(D) = 7-3 = 4$ because $D$, $\Omega(D)$ and $\Omega^2(D)$ are nontrivial but $\Omega^3(D)$ is trivial.
\end{example}

In view of Theorem~\ref{kIndepBound}, it will transpire that $b(D)$ is the cardinality of a $k$-independent set found by some application of MAX to some multigraph with degree sequence $D$ and we will show that this is the worst case.  We can prove the second part of Theorem~\ref{kIndepBound} (see Lemma \ref{lm:realisable}) without too much difficulty; to do so, we require the following lemma.

\begin{lemma}\label{lm:properties of Omega(D)}
Let $D$ be a nontrivial degree sequence.
\begin{itemize}
    \item[\textup{(i)}]
If $\O(D) = \{0,\dots,0\}$, then every reduction of $D$ is trivial.
    \item[\textup{(ii)}]
If $\O(D) \neq \{0,\dots,0\}$, then $\O(D)$ is a nontrivial degree sequence. Furthermore, for any multigraph $G'$ with degree sequence $\O(D)$, there is a multigraph $G$ such that deleting some vertex of maximum degree in $G$ results in $G'$.
\end{itemize}
\end{lemma}

\begin{proof}
Let $m=\max(D)$ and $m'=\max(\Omega(D))$. As in Definition~\ref{def:Omega}, let $(a_1,\ldots,a_s)$ be the decrement sequence of $D$, and let $A_0,\ldots,A_m$ be multisets such that $A_0=D\setminus\{m\}$, $A_m=\Omega(D)$, and $A_{i}$ is obtained from $A_{i-1}$ by an $a_i$-decrement for $i \in \{1,\ldots,m\}$.

First suppose that $\O(D) = \{0,\dots,0\}$. Then, by the definition of $\Omega$, either $\sum A_0 < m+2k$ or $\max(A_0)<k$. In the former case, every reduction of $D$ will be a degree sequence with sum less than $2k$ and hence will have maximum element less than $k$. In the latter case, every reduction of $D$ will clearly have maximum element less than $k$. Thus (i) is proved.

Now suppose that $\O(D) \neq \{0,\dots,0\}$. Then $\sum A_0 \geq m+2k$ and $\max(A_0) \geq k$. So $\sum \Omega(D) =  \sum A_0 - m \geq 2k$.  We first show that $m' \geq k$. Suppose for a contradiction that $m' < k$. Let $j \in \{1,2,\dots, m-1\}$ be such that $\max(A_j)
\geq k$ and $\max(A_{j+1}) < k$. Then, by the definition of a decrement sequence, a single element of $A_j$ is equal to $k$ and each other element of $A_j$ is equal to $0$. This contradicts $\sum A_{j} \geq \sum \O(D) \geq 2k$.

We next show that $\Omega(D)$ is a degree sequence.  Because $D$ is a degree sequence, $\sum D$ is even and hence $\sum \Omega(D) = \sum D - 2m$ is also even, so it remains to show that $\sum \O(D) \geq 2m'$.  Suppose otherwise for a contradiction that $\sum \O(D) < 2m'$. Then $2k \leq \sum \O(D) \leq 2m'-2$ because $\sum \Omega(D)$ is even. Hence $m' \geq k+1$ and $\max(\O(D)\setminus \{m'\}) \leq m'-2$. Thus, by the definition of a decrement sequence, $a_i > m'$ for $i \in
\{1,\ldots,m\}$ and hence $\max(A_0) = m'+m$. This contradicts $\max(A_0) \leq m$.

Finally let $G'$ be a multigraph with degree sequence $\Omega(D)$. We show that a multigraph $G$ with the claimed property exists. Let $H_m$ be the multigraph obtained from $G'$ by adding a new isolated vertex $u$ and note that $G'$ has degree sequence $A_m \uplus \{0\}$. We now inductively define multigraphs $H_m,H_{m-1},\ldots,H_0$ such that, for each $i \in \{m-1,m-2,\ldots,0\}$, $H_i$ is the multigraph with degree sequence $A_i \uplus \{m-i\}$ that is obtained from $H_{i+1}$ by adding an edge joining vertex $u$ and a vertex of degree $a_{i+1} -1$. (Note that $A_{i+1}$ contains an element equal to $a_{i+1} -1$ because it is obtained from $A_{i}$ by an $a_{i+1}$-decrement.) Let $G = H_0$.  Then $G$ has degree sequence $D$, $u$ is a vertex of maximum degree in $G$, and $G'=G-u$.
\end{proof}

\begin{example}\label{ex:build_graph}
Recall Example~\ref{ex:Omega} with $k=3$ and $D=\{1,2,2,4,4,5,6\}$.  Figure \ref{fig:build_graph2} shows four multigraphs whose degree sequences are, from left to right, $\Omega^3(D) = \{0,0,0,0\}$, $\Omega^2(D) = \{0,0,0,3,3\}$, $\Omega(D) = \{0,1,2,3,3,3\}$, and $D$. Each of the three leftmost multigraphs can be obtained from the multigraph to its immediate right by deleting some vertex of maximum degree.
\end{example}

\begin{figure}[h!]
\centering
 \includegraphics[width=0.9\linewidth]{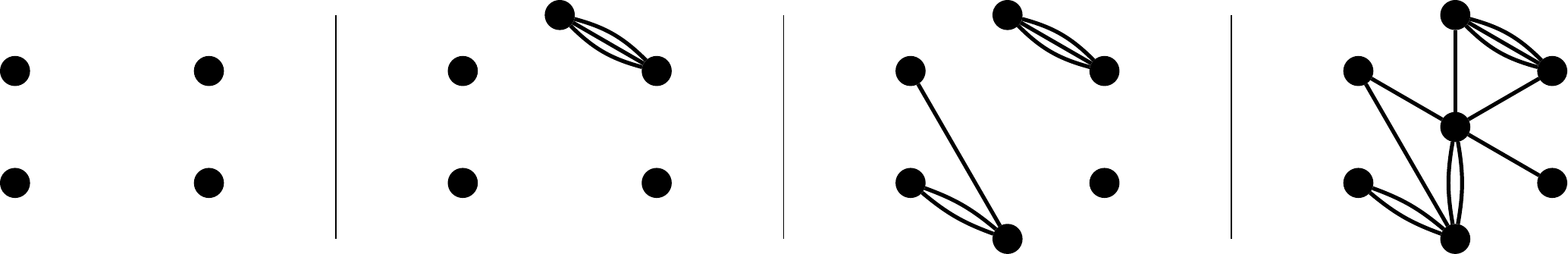}
\caption{Multigraphs with degree sequences $\Omega^3(D)$, $\Omega^2(D)$, $\Omega(D)$ and $D$ from Example~\ref{ex:build_graph}}
\label{fig:build_graph2}
\end{figure}

\begin{lemma}\label{lm:realisable}
Let $D$ be a degree sequence. There exists a multigraph $G$ such that some application of MAX to $G$ will result in a $k$-independent set of cardinality exactly $b(D)$.
\end{lemma}

\begin{proof}
Say $n$ is the order of $D$. Note that if $D$ is trivial, $b(D)=n$ and the result holds. Also, if $D$ is nontrivial but $\Omega(D)$ is trivial, then every reduction of $D$ is trivial by Lemma~\ref{lm:properties of Omega(D)}, so $b(D)=n-1$ and again the result can be seen to hold.  Thus, we assume that $\Omega(D)$ is nontrivial.

Suppose by induction that the result is true for all degree sequences of order less than $n$. By our inductive hypothesis, there exists a multigraph $G'$ with degree sequence $\Omega(D)$ such that some application of MAX to $G'$ results in a $k$-independent set of cardinality exactly $b(\Omega(D))$. By Lemma~\ref{lm:properties of Omega(D)}(ii), there is a multigraph $G$ with degree sequence $D$ such that deleting some vertex of maximum degree in $G$ results in $G'$. Then some application of MAX to $G$ produces a $k$-independent set of cardinality $b(\Omega(D))$ and from the definition of $b(D)$ we have that $b(D) = b(\Omega(D))$.
\end{proof}

The remainder of Theorem \ref{kIndepBound} is proved in Section \ref{s:mainproof}.

\section{Proof of Theorem \ref{kIndepBound} \label{s:mainproof}}

The goal of this section is to complete the proof of Theorem \ref{kIndepBound}.  The key idea of the proof involves a partial order $\preceq$ which we define on the set of all degree sequences of a given order.  The significance of this partial order stems from the fact that, if $D$ and $E$ are degree sequences with $D \preceq E$, then $b(D) \leq b(E)$. We will establish this fact in Lemma \ref{bbound}.

We require some further notation in order to define the partial order $\preceq$. Similar to the definition of an $x$-decrement, we now define an $x$-increment.

\begin{definition}[{\bf $\bm{x}$-increment}]
Let $E$ be a multiset of nonnegative integers and let $x$ be a positive element of $E$. We say that $D$ is obtained from $E$ by an {\it $x$-increment} if $D=(E \setminus \{x\}) \uplus \{x+1\}$.
\end{definition}

\begin{definition}[{\bf elementary step}]
Let $D$ and $E$ be multisets such that for some positive integers $x$ and $y$
\begin{itemize}
    \item[(i)]
$D$ is obtained from $E$ by first performing an $(x-1)$-increment and then performing a $(y-1)$-increment where $x \leq y \leq \max(E)+1$; or
    \item[(ii)]
$D$ is obtained from $E$ by first performing an $x$-decrement and then performing a $(y-1)$-increment where either $x > \max(k,y)$ or $x < y \leq k$.
\end{itemize}
We say that $D$ is obtained from $E$ by an \emph{elementary step}.  We call a
step of type (i) an \emph{$(x,y)$-addition step} and a step of type (ii) an
\emph{$(x,y)$-transfer step}.
\end{definition}

\begin{example}\label{ex:elementary}
Suppose that $k=3$, $E=\{1,2,2,4,4,5,6\}$ and $E^*=\{0,1,2,3,3,3\}$. Then the result of applying a $(3,7)$-addition step to $E$ is $D=\{1,2,3,4,4,5,7\}$ and the result of applying a $(1,3)$-transfer step to $E'$ is $D'=\{0,0,3,3,3,3\}$. Figure~\ref{fig:elementary} shows Ferrers diagrams for $E$, $E'$, $D$ and $D'$.  The addition step by which $D$ is obtained from $E$ is indicated by shaded boxes and the transfer step by which $D'$ is obtained from $E'$ is indicated by a dashed box and a shaded box. \end{example}

\begin{figure}[h!]
\centering
 \includegraphics[scale=0.5]{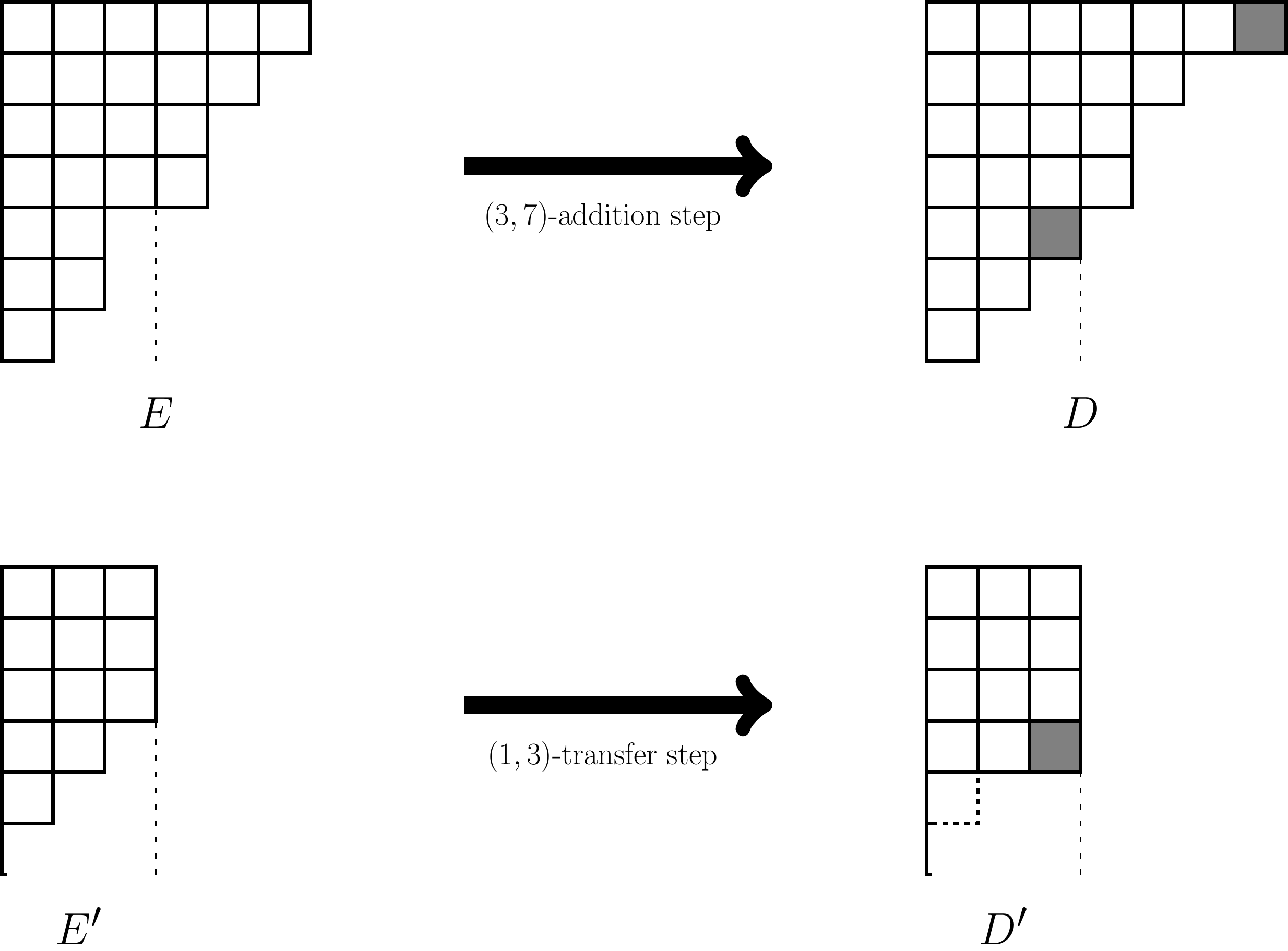}
\caption{Ferrers diagrams for $E$, $E'$, $D$ and $D'$ from Example~\ref{ex:elementary}}
\label{fig:elementary}
\end{figure}

\begin{definition}[\textbf{$\preceq$}] Let $D$ and $E$ be multisets with $n$ elements.  We say $D \prec E$ if $D$ can be obtained from $E$ by a (nontrivial) sequence of elementary steps,  and we say $D \preceq E$ if $D=E$ or $D \prec E$. It can be seen that $\preceq$ is a partial order on the set of multisets with $n$ elements.
\end{definition}

We now observe that elementary steps preserve the property of being a nontrivial degree sequence.

\begin{lemma}\label{lem:remain deg seq}
If $E$ is a nontrivial degree sequence and $D$ is a multiset such that $D \preceq E$, then $D$ is a nontrivial degree sequence.
\end{lemma}

\begin{proof}
By the transitivity of $\preceq$ it suffices to consider the case where $D$ is obtained from $E$ by an elementary step. Since $\max(E) \geq k$ and no elementary step involves performing a
$k$-decrement, $\max(D) \geq k$. Because $E$ is a degree sequence, $\sum E$ is even and $\sum E - 2\max(E) \geq 0$. It suffices to show that $\sum D$ is even and $\sum D - 2\max(D) \geq \sum E - 2\max(E)$ to establish that $D$ is a degree sequence and complete the proof.

If $D$ is obtained from $E$ by an $(x,y)$-addition step for some $x$ and $y$ then $\sum D = \sum E+2$ and, because $x,y \leq \max(E)+1$, $\max(D) \leq \max(E)+1$. If $D$ is obtained from $E$ by an $(x,y)$-transfer step for some $x$ and $y$, then $\sum D = \sum E$ and, because $y \leq \max(k,x-1)$, $\max(D) \leq \max(E)$. So in either case $\sum D$ is even and $\sum D - 2\max(D) \geq \sum E - 2\max(E)$.
\end{proof}

In order to prove Theorem~\ref{kIndepBound} we will require Lemma \ref{lem:main}  concerning the properties of reductions of degree sequences under the $\preceq$ order.  Since the proof of Lemma \ref{lem:main} requires additional concepts and notation, we defer it to Section \ref{s:lemmaproofs}.

\begin{lemma}\label{lem:main}
Let $D$ and $E$ be degree sequences such that $D \preceq E$, and let $E'$ be a nontrivial reduction of $E$. Then $\Omega(D) \preceq E'$.
\end{lemma}

\begin{lemma} \label{bbound}
Let $D$ and $E$ be degree sequences. If $D \preceq E$ then $b(D) \leq b(E)$.
\end{lemma}

\begin{proof}
Let $n$ be the order of $D$ and $E$. We have $b(E)= n-p$ where $p$ is the smallest nonnegative integer such that $\Omega^p(E)$ is trivial. So by Lemma~\ref{lm:properties of Omega(D)}, $\Omega^{i+1}(E)$ is a nontrivial reduction of $\Omega^i(E)$ for each $i \in \{0,\ldots,p-2\}$. Thus, by iteratively applying Lemma \ref{lem:main}, we have that $\Omega^i(D) \preceq \Omega^i(E)$ for each $i \in \{0,\ldots,p-1\}$. So by Lemma~\ref{lem:remain deg seq}, $\Omega^i(D)$ is nontrivial for each $i \in \{0,\ldots,p-1\}$ and hence $b(D) \leq n-p$.
\end{proof}

\begin{proof}[{\bf Proof of Theorem \ref{kIndepBound}}]
Say $n$ is the order of $D$. We may assume that $D$ is nontrivial, for otherwise $b(D)=n$ and the result holds. Let $G$ be an arbitrary multigraph with degree sequence $D$. In view of Lemma~\ref{lm:realisable}, we only need show that any application of MAX to $G$ results in a $k$-independent set of cardinality at least $b(D)$. Suppose that some application of MAX to $G$ results in a $k$-independent set of cardinality $n-p$ for some $p \in \{0,\ldots,n\}$. It suffices to show that $b(D) \leq n-p$. Say that the sequence of multigraphs obtained through the application of MAX is $G_0,G_1,\ldots,G_p$ where $G_0=G$, $G_{i}$ has maximum degree at least $k$ for $i \in \{0,\ldots,p-1\}$, and $G_p$ has maximum degree at most $k-1$. For $i \in \{0,\ldots,p\}$, let $D_i$ be the degree sequence of $G_i$. Observe that $D_{i+1}$ is a nontrivial reduction of $D_i$ for each $i \in \{0,\ldots,p-2\}$. Thus, by iteratively applying Lemma \ref{lem:main}, we have that $\Omega^i(D) \preceq D_i$ for each $i \in \{1,\dots,p-1\}$. So by Lemma~\ref{lem:remain deg seq}, $\Omega^i(D)$ is nontrivial for each $i \in \{0,\ldots,p-1\}$ and hence $b(D) \leq n-p$.
\end{proof}

\section{Proof of Lemma \ref{lem:main} \label{s:lemmaproofs}}

In this section we prove Lemma \ref{lem:main} as an easy consequence of two more technical lemmas, namely Lemmas \ref{lem:Omega_vs_other} and \ref{lem:OmegaStep}.  First we define some further notation which will be useful in these proofs.

\begin{definition}[$\bm{\mu_D, \sigma_D}$] Let $D$ be a finite multiset of nonnegative integers.  Define $\mu_D: \N \rightarrow \N$ and $\sigma_D: \N \rightarrow \N$ so that $\mu_D(z)$ is the number of elements (possibly 0) of $D$ equal to $z$, and $\sigma_D(z)$ is the number of elements of $D$ that are at least $z$.
\end{definition}

By definition, $\sigma_D(0)=|D|$ and $\mu_D(z) = \sigma_D(z)-\sigma_D(z+1)$ for each $z \in \N$.  Note that $\sigma_D$ is nonincreasing. If $D$ is viewed as an integer partition of $\sum D$ (where we allow parts equal to 0), then $\{\sigma_D(i)\}_{i \in \N}$ is the conjugate partition and, in the Ferrers diagram of $D$, the elements of $\{\sigma_D(i)\}_{i \in \N}$ are given by the column depths.

\begin{example}\label{ex:sigma}
If $D=\{0,1,1,3,3\}$, then $\sigma_D(0) = 5$, $\sigma_D(1)=4$,
$\sigma_D(2)=2$, $ \sigma_D(3) = 2$, and $\sigma_D(x) = 0$ for $x\geq 4$. Figure~\ref{fig:sigma} shows a Ferrers diagram of $D$.
\end{example}

\begin{figure}[h!]
\centering
\includegraphics[scale=0.5]{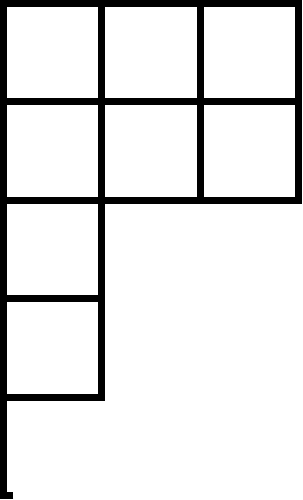}
\caption{Ferrers diagram for the degree sequence $D = \{0,1,1,3,3\}$ from Example~\ref{ex:sigma}}\label{fig:sigma}
\end{figure}

Given any nonincreasing function $f: \N \rightarrow \N$, there is exactly one finite multiset $D$ of nonnegative integers  such that $\sigma_D=f$ and so we can define a multiset $D$ by specifying $\sigma_D$. For each $x \in \N$ we define the indicator function $1_x
: \N \rightarrow \N$ by
\[1_x(z) =
\begin{cases}
0 & \mbox{if $z \neq x$;} \\
1 & \mbox{if $z = x$}.
\end{cases}
\]
Note that $D$ is obtained from $E$ by an $x$-decrement if $\sigma_{D}=\sigma_E-1_x$ and $D$ is obtained from $E$ by an $x$-increment if $\sigma_{D}=\sigma_E+1_{x+1}$.

\begin{lemma}\label{lem:Omega_vs_other}
Let $E$ be a degree sequence and let $E'$ be a nontrivial reduction of $E$.  Then $\Omega(E) \preceq E'$.
\end{lemma}

\begin{proof}
Fix a degree sequence $E$ and let $A_0 = E \setminus \{\max(E)\}$. Call any degree sequence $E'$ such that $\sum E'= \sum A_0 - \max(E)$ and  $\sigma_{E'}(z) \leq \sigma_{A_0}(z)$ for each positive integer $z$ a \emph{pseudo-reduction} of $E$. Clearly every reduction of $E$ is also a pseudo-reduction. We will in fact prove the lemma for any nontrivial pseudo-reduction $E'$ of $E$. It is clear that $E$ has only finitely many pseudo-reductions.

Let $E'$ be a pseudo-reduction of $E$. If $E' = \Omega(E)$, then the result is immediate. So suppose that $E' \neq \Omega(E)$ and, by induction, that $\Omega(E) \preceq E''$ for any pseudo-reduction $E''$ of $E$ such that $E'' \prec E'$. Let $x$ be the largest positive integer such that $\sigma_{E'}(x)>\sigma_{\Omega(E)}(x)$ and let $y$ be the smallest positive integer
such that $\sigma_{E'}(y)<\sigma_{\Omega(E)}(y)$. Such integers exist because $E'
\neq \Omega(E)$ and $\sum E'= \sum \Omega(E)$. Note that $x \neq y$.

We will show that we can perform an $(x,y)$-transfer step to $E'$ to obtain a degree sequence $E''$, say. Then $E'' \prec E'$ obviously and $E''$ will be a pseudo-reduction of $E$ because $\sigma_{E''}(y) \leq \sigma_{\Omega(E)}(y) \leq \sigma_{A_0}(y)$, and  $\sigma_{E''}(z) \leq
\sigma_{E'}(z) \leq \sigma_{A_0}(z)$ for each positive integer $z \neq y$. So
$\Omega(E) \preceq E''$ by our inductive hypothesis and the result will follow by the transitivity of $\preceq$.

It only remains to show that we can perform an $(x,y)$-transfer step on $E'$.
Let $A_0,\ldots,A_{\max(E)}$ be the sequence of multisets as given in Definition \ref{def:Omega}, where $A_{\max(E)}=\Omega(E)$. Let $i$ be the element of $\{0,\ldots,\max(E)-1\}$ such that $\sigma_{A_{i+1}}(x)=\sigma_{\Omega(E)}(x)$ but $\sigma_{A_{i}}(x)=\sigma_{\Omega(E)}(x)+1$. Note $i$ exists because $\sigma_E(x) \geq \sigma_{E'}(x) > \sigma_{\Omega(E)}(x)$. Obviously, $A_{i+1}$ is obtained from $A_{i}$ by an $x$-decrement. Furthermore, by the definition of a decrement sequence, $A_i$ contains no element greater than $\max(k,x)$ and, if $x \leq k$, $A_i$ has exactly $\sigma_{\Omega(E)}(x)+1$ positive elements. We make and prove three claims.

\begin{itemize}
    \item[(i)]
\textbf{$\bm{y < x}$ if $\bm{x>k}$, and $\bm{y \leq k}$ if $\bm{x \leq k}$.} Because $A_i$ contains no element greater than $\max(k,x)$, nor does $\Omega(E)$. So $\sigma_{\Omega(E)}(z)=0$ for each $z > \max(k,x)$ and (i) follows from the definition of $y$.
    \item[(ii)]
\textbf{$\bm{x < y}$ if $\bm{x \leq k}$.} Suppose for a contradiction that $y < x \leq k$. Because $x \leq k$, $A_i$ has exactly $\sigma_{\Omega(E)}(x)+1$ positive elements and hence $\Omega(E)$ has at most $\sigma_{\Omega(E)}(x)+1$ positive elements. However,
\[\sigma_{\Omega(E)}(y) > \sigma_{E'}(y) \geq \sigma_{E'}(x) \geq \sigma_{\Omega(E)}(x)+1\]
where the inequalities hold by, respectively, the definition of $y$, the fact that $\sigma$ is a nondecreasing function, and the definition of $x$. This contradicts the fact that $\Omega(E)$ has at most $\sigma_{\Omega(E)}(x)+1$ positive elements.
    \item[(iii)] \textbf{$\bm{x,y-1 \in E'}$, and $\bm{\mu_D(x) \geq 2}$ if $\bm{x=y-1}$.} This follows because
\[
\begin{array}{rccccccl}
\sigma_{E'}(x) &>& \sigma_{\Omega(E)}(x) &\geq& \sigma_{\Omega(E)}(x+1) &\geq& \sigma_{E'}(x+1)& \mbox{and} \\
  \sigma_{E'}(y-1) &\geq& \sigma_{\Omega(E)}(y-1) &\geq& \sigma_{\Omega(E)}(y) &>& \sigma_{E'}(y)&
\end{array}
\]
where in each case the leftmost and rightmost hold by our definitions of $x$ and
$y$ and the middle inequality holds by monotonicity of $\sigma_{\O(E)}$. When
$x=y-1$, $\sigma_{E'}(x) > \sigma_{\Omega(E)}(x) > \sigma_{E'}(x+1)$.
\end{itemize}

Together (i) and  (ii) imply that either $x > \max(k,y)$ or $x < y \leq k$. So by (i), (ii) and (iii) we can indeed perform an $(x,y)$-transfer step on $E'$.
\end{proof}

In proving Lemma~\ref{lem:OmegaStep} below we will show that if $D$ and $E$ are degree sequences such that $D$ is obtained from $E$ by an elementary step, then $\Omega(D) \preceq \Omega(E)$ provided that $\Omega(E)$ is nontrivial. We will first give an example of this and then state and prove Lemma~\ref{lem:OmegaStep}.

\begin{example}\label{ex:Omega_vs_other}
Recall Example~\ref{ex:elementary} where $k=3$. There, $D$ was obtained from $E$ by an elementary step. Now observe that $E'=\Omega(E)$ and $D'=\Omega(D)$. So the fact that $D'$ can be obtained from $E'$ by an elementary step establishes that $\Omega(D) \preceq \Omega(E)$. In fact, the degree sequences in Example~\ref{ex:elementary} will be an instance of Subcase 1A in the proof of Lemma~\ref{lem:OmegaStep}.
\end{example}

\begin{lemma}\label{lem:OmegaStep}
Let $D$ and $E$ be degree sequences such that $D \preceq E$ and $\Omega(E)$ is nontrivial. Then $\Omega(D) \preceq \Omega(E)$.
\end{lemma}

\begin{proof}
By the definition of $\preceq$ and its transitivity, it suffices to consider the case where
$D$ is obtained from $E$ by a single elementary step. Note that by Lemma~\ref{lm:properties of Omega(D)}, both $\O(D)$ and $\O(E)$ are degree sequences. Let $A_0 =D \setminus \{\max(D)\}$ and $B_0 =E \setminus \{\max(E)\}$.

Let $s = \sum A_0$, $t = \sum B_0$ and let $(a_1,\ldots,a_s)$ and $(b_1,\ldots,b_t)$ be the decrement sequences of $D$ and $E$, respectively. Let $\tau_A,\tau_B: \N \rightarrow \N$ be the functions
defined by
\begin{align*}
  \tau_A(x) &= |\{i:a_i=x,1 \leq i\leq \max(D)\}| \mbox{ and} \\
  \tau_B(x) &= |\{i:b_i=x,1 \leq i\leq \max(E)\}|.
\end{align*}
Note that $\sigma_{\Omega(D)}=\sigma_{A_0}-\tau_A$ and $\sigma_{\Omega(E)}=\sigma_{B_0}-\tau_B$.  For conciseness, let $m = \max(E)$.

The proof divides into cases, depending on whether $D$ is obtained from $E$ by an addition step or transfer step.  In each case, by analysing $\tau_A$ and $\tau_B$ we show that $\sigma_{\Omega(D)} = \sigma_{\Omega(E)}$, or $\sigma_{\Omega(D)} = \sigma_{\Omega(E)}+1_x+1_y$ where $x \leq y \leq m+1$, or $\sigma_{\Omega(D)} = \sigma_{\Omega(E)}-1_x + 1_y $ where either $x > \max(k,y)$ or $x<y\leq k$.  Since $\Omega(D)$ and $\Omega(E)$ are degree sequences, it follows that $\Omega(D)=\Omega(E)$ or $\Omega(D)$ is obtained from $\Omega(E)$ by a single elementary step, and thus $\Omega(D) \preceq \Omega(E)$. \\

\noindent {\bf Case 1.} Suppose $D$ is obtained from $E$ by an addition step.  In this case $\sigma_{D}=\sigma_E+1_{x}+1_{y}$ where
$x,y \in \{1, \dots, m+1\}$ and $x \leq y$. We consider two subcases according to whether $y = m+1$.

\noindent {\bf Subcase 1A.}
Suppose $y=m+1$. Then $\max(D) = m+1$, $s = t+1$ and $\sigma_{A_0}=\sigma_{B_0}+1_x$. Let $r$ be the
smallest element of $\{1,\dots, s\}$ such that $a_r=x$.
Now, for $i \in \{1,\dots,s\}$ we have
\[a_i=
\begin{cases}
b_i & \mbox{if $i < r$;} \\
x & \mbox{if $i = r$;} \\
b_{i-1} & \mbox{if $i > r$.} \\
\end{cases}
\]
It follows that $\tau_A$ takes the values given below.  Substituting these values for $\tau_A$ as well as $\sigma_{A_0}=\sigma_{B_0}+1_x$ into $\sigma_{\O(D)}=\sigma_{A_0}-\tau_A$, we obtain the following values for
$\sigma_{\Omega(D)}$.
\begin{center}
\begin{tabular}{l|l|l}
  case & $\tau_A$ & $\sigma_{\Omega(D)}$ \\ \hline
  $m+1 < r$ & $\tau_B + 1_{b_{m+1}}$ & $\sigma_{\Omega(E)}-1_{b_{m+1}}+1_x$ \\
  $m+1 \geq r$ & $\tau_B+1_x$ &
$\sigma_{\Omega(E)}$ \\
\end{tabular}
\end{center}
It remains only to show that if $m+1 < r$ then $b_{m+1}$ and $x$ satisfy the conditions in the definition of a $(b_{m+1},x)$-transfer step. Since $a_{m+1}=b_{m+1}$, by the definition of $r$, it follows that $x \not= b_{m+1}$ and a $b_{m+1}$ occurs before the first $x$ in the decrement sequence $(a_1,\ldots,a_s)$. Thus, by the definition of a decrement sequence, either $b_{m+1} > \max(k,x)$ or $1 \leq b_{m+1} < x \leq k$ as required.

\noindent {\bf Subcase 1B.}
Suppose $y \leq m$. Then $\max(D)=m$, $s=t+2$, and $\sigma_{A_0}=\sigma_{B_0}+1_x+1_y$. Let $q$ be the smallest
element in $\{1,\dots, s\}$ such that $a_{q} \in \{x,y\}$.
Denote $x' = a_{q}$ and $y' = \{x,y\} \setminus \{x'\}$. Let $r$ be the smallest
element in $\{q+1,\dots, s\}$ such that $a_{r} =y'$. Then, for $i \in \{1,\dots,s\}$,
\[a_i=
\begin{cases}
b_i & \mbox{if $i < q$;} \\
x' & \mbox{if $i = q$;}\\
b_{i-1} & \mbox{if $q \leq i < r$;}\\
y' & \mbox{if $i = r$;}\\
b_{i-2} & \mbox{if $i > r$.} \\
\end{cases}
\]
It follows that $\tau_A$ takes the values given below.  Substituting these values for $\tau_A$ and $\sigma_{A_0}=\sigma_{B_0}+1_x+1_y$ into $\sigma_{\Omega(D)}=\sigma_{A_0}-\tau_A$, we obtain the following values for
$\sigma_{\Omega(D)}$.
\begin{center}
\begin{tabular}{l|l|l}
  case & $\tau_A$ & $\sigma_{\Omega(D)}$ \\ \hline
  $m < q$ & $\tau_B$ & $\sigma_{\Omega(E)}+1_x+1_y$ \\
  $q \leq m < r$ & $\tau_B+1_{x'}-1_{b_m}$ &
$\sigma_{\Omega(E)}+1_{y'}+1_{b_m}$ \\
  $m \geq r$ & $\tau_B+1_x+1_y-1_{b_{m-1}}-1_{b_m}$ &
$\sigma_{\Omega(E)}+1_{b_{m-1}}+1_{b_m}$ \\
\end{tabular}
\end{center}
So in each case $\Omega(D)$ can be obtained from $\Omega(E)$ by an addition step because $b_{m-1},b_{m} \leq m$ and, by assumption, $x \leq y \leq m+1$. \\

\noindent {\bf Case 2.} Suppose $D$ is obtained from $E$ by a transfer step.  In this case $\sigma_{D}=\sigma_E-1_{x}+1_{y}$ where either $x > \max(k,y)$ or $x < y \leq k$.  We consider two subcases according to whether $\max(D) = m$.

\noindent {\bf Subcase 2A.} Suppose $\max(D)=m$.   Then $s=t$ and $\sigma_{A_0}=\sigma_{B_0}-1_x+1_y$. Let $q$ be the smallest element of $\{1,\ldots,s\}$ such that $b_q=x$, and let $r$
be the smallest element of $\{1,\ldots,s\}$ such that $a_r=y$. Then, for $i \in \{1,\ldots,s\}$,
\[a_i=
\begin{cases}
b_i & \mbox{if $i < q$ or $i> r$;} \\
b_{i+1} & \mbox{if $q \leq i < r$;} \\
y & \mbox{if $i=r$}.
\end{cases}
\]
Note that since either $x > \max(k,y)$ or $x < y \leq k$, we have $q \leq r$ by the definition of a decrement sequence.

It follows that $\tau_A$ takes the values given below.  Substituting these values for $\tau_A$ and $\sigma_{A_0}=\sigma_{B_0}-1_x+1_y$ into $\sigma_{\Omega(D)}=\sigma_{A_0}-\tau_A$, we obtain the following values for $\sigma_{\Omega(D)}$.
\begin{center}
\begin{tabular}{l|l|l}
  case & $\tau_A$ & $\sigma_{\Omega(D)}$ \\ \hline
  $m \geq r$ & $\tau_B -1_x+1_y$ & $\sigma_{\Omega(E)}$ \\
  $q \leq m < r$ & $\tau_B-1_x+1_{b_{m+1}}$ &
$\sigma_{\Omega(E)}-1_{b_{m+1}}+1_y$ \\
  $m < q$ & $\tau_B$ & $\sigma_{\Omega(E)}-1_x+1_y$ \\
\end{tabular}
\end{center}

Since $x > \max(k,y)$ or $1 \leq x < y \leq k$ by assumption, it only remains to show that if $q \leq m < r$ then $b_{m+1}$ and $y$ satisfy the conditions of a $(b_{m+1}, y)$-transfer step. Since $a_{m}=b_{m+1}$, by the definition of $r$, it follows that $y \not= b_{m+1}$ and a $b_{m+1}$ occurs before the first $y$ in the decrement sequence $(a_1,\ldots,a_s)$. Thus, by the definition of a decrement sequence, either $b_{m+1} > \max(k,y)$ or $1 \leq b_{m+1} < y \leq k$ as required.

\noindent {\bf Subcase 2B.} Suppose $\max(D) \not= m$.
Then $x=m$, $y<m$, $\max(D)=m-1$, $s = t-1$ and $\sigma_{A_0}=\sigma_{B_0}+1_y$.  Let $r$ be the smallest
element of $\{1,\ldots,s\}$ such that $a_r=y$. Then, for $i
\in \{1,\ldots,s\}$,
\[a_i=
\begin{cases}
b_i & \mbox{if $i < r$;\phantom{ or $i> r$}} \\
y & \mbox{if $i=r$;} \\
b_{i-1} & \mbox{if $i > r$.} \\
\end{cases}
\]
It follows that $\tau_A$ takes the values given below.  Substituting these values for $\tau_A$ and $\sigma_{A_0}=\sigma_{B_0}-1_x+1_y$ into $\sigma_{\Omega(D)}=\sigma_{A_0}-\tau_A$, we obtain the following values for $\sigma_{\Omega(D)}$.
\begin{center}
\begin{tabular}{l|l|l}
  case & $\tau_A$ & $\sigma_{\Omega(D)}$ \\ \hline
  $m -1 \leq r$ & $\tau_B-1_{b_m}$ & $\sigma_{\Omega(E)} + 1_y+1_{b_m}$
\\
  $m-1 > r$ & $\tau_B+1_y-1_{b_{m-1}}-1_{b_m}$ &
$\sigma_{\Omega(E)}+1_{b_{m-1}}+1_{b_m}$ \\
\end{tabular}
\end{center}
So in each case $\Omega(D)$ can be obtained from $\Omega(E)$ by an addition step because $y, b_{m-1}, b_m \leq m$.
\end{proof}

\vspace{2mm}

\begin{proof}[{\bf Proof of Lemma \ref{lem:main}}]
By Lemma~\ref{lem:Omega_vs_other}, $\Omega(E) \preceq E'$. By Lemma~\ref{lem:OmegaStep}, $\Omega(D) \preceq \Omega(E)$. So the result follows by the transitivity of $\preceq$.
\end{proof}

\section{An application to pair coverings \label{s:coverings}}

For positive integers $v$, $\kappa$ and $\lambda$, a {\it $(v,\kappa,\lambda)$-covering} is a pair $(V, \mathcal{B})$ where $V$ is a set of $v$ elements, called {\it points}, and $\mathcal{B}$ is a collection of $\kappa$-subsets of $V$, called {\it blocks}, such that each pair of points occurs together in at least $\lambda$ blocks.  A $(v,\kappa,\lambda)$-covering is a natural generalisation of a {\it $(v,\kappa,\lambda)$-design}, where each pair of points occurs together in exactly $\lambda$ blocks.  Coverings are well-studied combinatorial objects, see for example \cite{GoSt}. The case $\lambda=1$ is of particular interest.  Typically, one is interested in finding coverings with as few blocks as possible; the {\it covering number} $C_{\lambda}(v,\kappa)$ is the minimum number of blocks in any $(v,\kappa,\lambda)$-covering.  The {\it Sch\"{o}nheim bound} \cite{Sc} states that
\begin{equation*}\label{eq:Sbound}
 C_{\lambda}(v,\kappa) \geq \Bigl\lceil\frac{vr}{\kappa}\Bigr\rceil, \text{ where } r = \Bigl\lceil\frac{\lambda(v-1)}{\kappa-1}\Bigr\rceil.
 \end{equation*}

Improvements on the Sch\"{o}nheim bound have been made in various cases, see \cite{Ho} and the references therein. Exact covering numbers are known for $\kappa \in \{3,4\}$ and, when $\lambda=1$, for $v \leq \frac{13}{4}\kappa$ \cite{GoSt}. An online repository of coverings is maintained by Gordon \cite{Go}.

Let $\mathcal{C}$ be a $(v,\kappa,\lambda)$-covering on point set $V$.  For each $u \in V$, define $r_{\mathcal{C}}(u)$ to be the number of blocks of $\mathcal{C}$ that contain $u$.  Similarly, for all distinct $u,w \in V$, define $r_{\mathcal{C}}(uw)$ to be the number of blocks of $\mathcal{C}$ that contain both $u$ and $w$. We define the {\it excess} of $\mathcal{C}$ to be the multigraph $G$ on vertex set $V$ where, for all distinct $u,w \in V$, the multiplicity of edge $uw$ is $r_{\mathcal{C}}(uw) - \lambda$.  Observe that the excess of a covering is a loopless multigraph. The following is an immediate consequence of \cite[Lemma 6]{Ho}.

\begin{lemma} \textup{\cite{Ho}} \label{lem:m_indep_covering}
Let $v$, $\kappa$ and $\lambda$ be positive integers such that $3 \leq \kappa < v$, let $\mathcal{C}$ be a $(v,\kappa,\lambda)$-covering on point set $V$ and let $G$ be the excess of $\mathcal{C}$.  If there is a subset $S \subseteq V$ such that $S$ is an $(r'-\lambda)$-independent set in $G$, where $r' = \min(\{ r_{\mathcal{C}}(u) : \, u \in S\})$, then $\mathcal{C}$ had at least $|S|$ blocks.
\end{lemma}

In \cite{Ho}, Caro and Tuza's bound on independence number from \cite{CaTu} was used together with Lemma \ref{lem:m_indep_covering} to establish new lower bounds on covering numbers in the case where the block size is a significant fraction of the number of points. By employing Theorem~\ref{kIndepBound} instead of the bound of Caro and Tuza, we can establish the following.

\begin{theorem} \label{thm:coverbound}
Let $v$, $\kappa$ and $\lambda$ be positive integers such that $3 \leq \kappa < v$, and let $r$ and $d$ be the integers such that $\lambda(v-1) = r(\kappa-1) -d$ and $0 \leq d < \kappa-1$. If there exists a $(v,\kappa,\lambda)$-covering with $z$ blocks, then $z \geq b_{r-\lambda}(D)$,
where
\begin{itemize}
    \item
$D$ is the degree sequence of order $v$ with $\ell$ elements equal to $d+(s+1)(\kappa-1)$ and $v-\ell$ elements equal to $d+s(\kappa-1)$; and
    \item
$s$ and $\ell$ are the nonnegative integers such that $\kappa z=(r+s)v+\ell$ and $0 \leq \ell < v$.
\end{itemize}
\end{theorem}

\begin{proof}
By considering the pairs of points involving a specified point in a $(v,\kappa,\lambda)$-covering with $z$ blocks, it can be deduced that each point is in at least $r$ blocks and hence that $\kappa z \geq rv$. From this fact it follows that $s$, $\ell$ and $D$ are well defined. Let $\mathcal{C}$ be a $(v,\kappa,\lambda)$-covering on point set $V$, let $G$ be the excess of $\mathcal{C}$, and let $E$ be the degree sequence of $G$. By Theorem~\ref{kIndepBound}, $G$ has a $(r-\lambda)$-independent set of cardinality at least $b_{r-\lambda}(E)$ and hence by Lemma~\ref{lem:m_indep_covering}, $z \geq b_{r-\lambda}(E)$. So it suffices to show that $b_{r-\lambda}(E) \geq b_{r-\lambda}(D)$. By Lemma~\ref{bbound} then, it in fact suffices to prove that $D \preceq E$.

If $E=D$ the result is trivial, so assume otherwise. For any $u \in V$, the degree of $u$ in the excess is $d+(r_{\mathcal{C}}(u)-r)(\kappa-1)$ and hence $E$ is the multiset $\{d+(r_{\mathcal{C}}(u)-r)(\kappa-1):u \in V\}$. Note that $\sum_{u \in V}(r_{\mathcal{C}}(u)-r)=\kappa z-rv=sv+\ell$. Begin with $E$ and iteratively apply $(x,y)$-transfer steps, each time choosing $x$ to equal some element greater than $d+(s+1)(\kappa-1)$ and $y$ to equal the greatest element less than $d+(s+1)(\kappa-1)$, until no elements greater than $d+(s+1)(\kappa-1)$ remain. This process will terminate in a degree sequence equal to $D$.
\end{proof}

For the applications of Theorem~\ref{thm:coverbound} that we detail in this section, it is always the case that $\kappa z <v(r+1)$. Thus, $s=0$ and $D$ has $\ell$ elements equal to $d+\kappa-1$ and $v-\ell$ elements equal to $d$.

We can use Theorem~\ref{thm:coverbound} to attempt to improve any existing lower bound on a covering number. We set $z$ to be the existing bound and, if Theorem~\ref{thm:coverbound} gives a contradiction to the existence of a covering with $z$ blocks, we can conclude that the covering number is at least $z+1$. This procedure can then be iterated.

\begin{example}
A $(50,14,1)$-covering has at least $16$ blocks by \cite[Theorem 9]{Ho}, which is already an improvement over the Sch\"{o}nheim bound of $15$. When we apply Theorem~\ref{thm:coverbound} with $z=16$, we have $r = 4$, $d = 3$, and $D$ having $24$ elements equal to $16$ and $26$ elements equal to $3$. It can be calculated that $b_3(D)=17$ (note $r-\lambda=3$), which is greater than $z=16$. Thus Theorem~\ref{thm:coverbound} shows that a $(50,14,1)$-covering with $16$ blocks cannot exist and we conclude that $C_1(50,14) \geq 17$.
\end{example}

Recall that exact covering numbers are known for $\kappa \in \{3,4\}$ and, when $\lambda=1$, for $v \leq \frac{13}{4}\kappa$. Also, it is clear that Theorem~\ref{thm:coverbound} cannot produce bounds greater than $v$ and hence cannot improve the Sch\"{o}nheim bound when $r \geq \kappa$ or, equivalently, when $v > \frac{1}{\lambda}(\kappa-1)^2+1$.  For $\lambda = 1$, for each $\kappa \in \{5, \dots, 40\}$ and each integer $v$ such that $\frac{13\kappa}{4} < v \leq (\kappa-1)^2+1$, we record in Table \ref{tab:improvedBoundExamples} those parameters for which Theorem~\ref{thm:coverbound} yields an improvement on the best previously known bound. For each parameter set for which we see an improvement we give the best previously known bound, the source of that bound, the values of $d$ and $r$ when Theorem~\ref{thm:coverbound} is applied, and the new bound yielded. In none of the situations detailed does a second application of Theorem~\ref{thm:coverbound} further improve the bound, and so the new bounds listed are always exactly one more than the previous bounds. It is worth noting that in all of the situations for which we obtain improvements, $d>r$ and the best previously known bound is given by either the Sch\"{o}nheim bound or \cite[Theorem 9]{Ho}.

The coverings discussed in this section are in fact pair coverings: the special case of $t$-$(v,\kappa,\lambda)$ coverings with $t=2$ (see \cite{GoSt} for the relevant definitions). Through using \cite[Lemma 2.5]{HorSin} in place of Lemma~\ref{lem:m_indep_covering}, we can attempt to improve lower bounds on the number of blocks in $t$-$(v,\kappa,\lambda)$ coverings with $t \geq 3$ by similar means to those described above for $t=2$. Our preliminary computations found only a few such improvements (for parameter sets $3$-$(93,42,1)$, $5$-$(67,41,1)$, and $6$-$(41,26,1)$), so we do not attempt a more systematic study here.

\begin{table}[h!]
\centering
\scriptsize
\begin{tabular}{cc|ccc|c|c|c}
$\kappa$  & $v$ & $d$ & $r$ & $\ell$ & best previous bound & source of previous bound & new bound \\
\hline
$14$ & $50$ & $3$ & $4$ & $24$ & $16$ & \cite[Theorem 9]{Ho} & $17$\\
$16$ & $56$ & $5$ & $4$ & $16$ & $15$ & \cite[Theorem 9]{Ho} & $16$\\
$17$ & $61$ & $4$ & $4$ & $28$ & $16$ & \cite[Theorem 9]{Ho} & $17$\\
$19$ & $155$ & $8$ & $9$ & $11$ & $74$ & Sch\"{o}nheim bound & $75$\\
$20$ & $72$ & $5$ & $4$ & $32$ &  $16$ & \cite[Theorem 9]{Ho} & $17$\\
$21$ & $115$ & $6$ & $6$ & $45$ &  $35$ & \cite[Theorem 9]{Ho}  & $36$\\
$21$ & $192$ & $9$ & $10$ & $12$ & $92$ & Sch\"{o}nheim bound & $93$\\
$22$ & $102$ & $4$ & $5$ & $62$ & $26$ & \cite[Theorem 9]{Ho} & $2$\\
$22$ & $117$ & $10$ & $6$ & $2$ & $32$ & Sch\"{o}nheim bound & $33$\\
$22$ & $139$ & $9$ & $7$ & $17$ & $45$ & Sch\"{o}nheim bound & $46$\\
$22$ & $140$ & $8$ & $7$ & $32$ &  $46$ & \cite[Theorem 9]{Ho}  & $47$\\
$22$ & $141$ & $7$ & $7$ & $47$ &  $47$ & \cite[Theorem 9]{Ho}  & $48$\\
$22$ & $142$ & $6$ & $7$ & $62$ &  $48$ & \cite[Theorem 9]{Ho} & $49$\\
$23$ & $83$ & $6$ & $4$ & $36$ &  $16$ & \cite[Theorem 9]{Ho}  & $17$\\
$24$ & $128$ & $11$ & $6$ & $0$ & $32$ & Sch\"{o}nheim bound & $33$\\
$24$ & $152$ & $10$ & $7$ & $16$ & $45$ & Sch\"{o}nheim bound & $46$\\
$24$ & $174$ & $11$ & $8$ & $0$ & $58$ & Sch\"{o}nheim bound & $59$\\
$25$ & $163$ & $6$ & $7$ & $84$ & $49$ & \cite[Theorem 9]{Ho}  & $50$\\
$25$ & $208$ & $9$ & $9$ & $53$ &  $77$ & \cite[Theorem 9]{Ho} & $78$\\
$26$ & $94$ & $7$ & $4$ & $40$ &  $16$ & \cite[Theorem 9]{Ho}  & $17$\\
$26$ & $114$ & $12$ & $5$ & $2$ & $22$ & Sch\"{o}nheim bound & $23$\\
$26$ & $143$ & $8$ & $6$ & $52$ & $35$ & \cite[Theorem 9]{Ho}  & $36$\\
$26$ & $290$ & $11$ & $12$ & $30$ & $135$ & \cite[Theorem 9]{Ho}  & $136$\\
$27$ & $122$ & $9$ & $5$ & $38$ & $24$ & \cite[Theorem 9]{Ho} & $25$\\
$28$ & $99$ & $10$ & $4$ & $24$ & $15$ & Sch\"{o}nheim bound & $16$\\
$28$ & $123$ & $13$ & $5$ & $1$ & $22$ & Sch\"{o}nheim bound & $23$\\
$28$ & $285$ & $13$ & $11$ & $1$ & $112$ & Sch\"{o}nheim bound & $113$\\
$29$ & $105$ & $8$ & $4$ & $44$ &  $16$ & \cite[Theorem 9]{Ho}  & $17$\\
$30$ & $132$ & $14$ & $5$ & $0$ & $22$ & Sch\"{o}nheim bound & $23$\\
$30$ & $167$ & $8$ & $6$ & $78$ &  $36$ & \cite[Theorem 9]{Ho}  & $37$\\
$30$ & $335$ & $14$ & $12$ & $0$ & $134$ & Sch\"{o}nheim bound & $135$\\
$31$ & $171$ & $10$ & $6$ & $59$ &  $35$ & \cite[Theorem 9]{Ho}  & $36$\\
$31$ & $257$ & $14$ & $9$ & $2$ & $75$ & Sch\"{o}nheim bound & $76$\\
$31$ & $287$ & $14$ & $10$ & $13$ & $93$ & Sch\"{o}nheim bound & $94$\\
$32$ & $116$ & $9$ & $4$ & $48$ &  $16$ & \cite[Theorem 9]{Ho}  & $17$\\
$32$ & $143$ & $13$ & $5$ & $21$ & $23$ & Sch\"{o}nheim bound & $24$\\
$32$ & $237$ & $12$ & $8$ & $56$ &  $61$ & \cite[Theorem 9]{Ho}  & $62$\\
$33$ & $242$ & $15$ & $8$ & $11$ & $59$ & Sch\"{o}nheim bound & $60$\\
$33$ & $275$ & $14$ & $9$ & $33$ &  $76$ & \cite[Theorem 9]{Ho} & $77$\\
$33$ & $370$ & $15$ & $12$ & $15$ & $135$ & Sch\"{o}nheim bound & $136$\\
$34$ & $152$ & $14$ & $5$ & $22$ & $23$ & Sch\"{o}nheim bound & $24$\\
$34$ & $186$ & $13$ & $6$ & $40$ &  $34$ & \cite[Theorem 9]{Ho}  & $35$\\
$34$ & $352$ & $12$ & $11$ & $106$ &  $117$ & \cite[Theorem 9]{Ho} & $118$\\
$35$ & $124$ & $13$ & $4$ & $29$ & $15$ & Sch\"{o}nheim bound & $16$\\
$35$ & $127$ & $10$ & $4$ & $52$ &  $16$ & \cite[Theorem 9]{Ho}  & $17$\\
$35$ & $298$ & $9$ & $9$ & $153$ &  $81$ & \cite[Theorem 9]{Ho} & $82$\\
$36$ & $161$ & $15$ & $5$ & $23$ & $23$ & Sch\"{o}nheim bound & $24$\\
$36$ & $199$ & $12$ & $6$ & $66$ & $35$ & \cite[Theorem 9]{Ho} & $36$\\
$36$ & $406$ & $15$ & $12$ & $60$ & $137$ & \cite[Theorem 9]{Ho} & $138$\\
$37$ & $168$ & $13$ & $5$ & $48$ &  $24$ & \cite[Theorem 9]{Ho} & $25$\\
$38$ & $138$ & $11$ & $4$ & $56$ &  $16$ & \cite[Theorem 9]{Ho} & $17$\\
$38$ & $141$ & $8$ & $4$ & $82$ &  $17$ & \cite[Theorem 9]{Ho} & $18$\\
$38$ & $246$ & $14$ & $7$ & $64$ &  $47$ & \cite[Theorem 9]{Ho}  & $48$\\
$38$ & $614$ & $16$ & $17$ & $88$ &  $277$ & \cite[Theorem 9]{Ho} & $278$\\
$39$ & $220$ & $9$ & $6$ & $123$ &  $37$ & \cite[Theorem 9]{Ho}  & $38$\\
$39$ & $366$ & $15$ & $10$ & $84$ & $96$ & \cite[Theorem 9]{Ho}  & $97$\\
$39$ & $591$ & $18$ & $16$ & $21$ & $243$ & Sch\"{o}nheim bound & $244$\\
$40$ & $142$ & $15$ & $4$ & $32$ & $15$ & Sch\"{o}nheim bound & $16$\\
$40$ & $187$ & $9$ & $5$ & $105$ &  $26$ & \cite[Theorem 9]{Ho} & $27$\\
$40$ & $302$ & $11$ & $8$ & $144$ &  $64$ & \cite[Theorem 9]{Ho}  & $65$\\
$40$ & $372$ & $19$ & $10$ & $0$ & $93$ & Sch\"{o}nheim bound & $94$\\
$40$ & $412$ & $18$ & $11$ & $28$ & $114$ & Sch\"{o}nheim bound & $115$\\
$40$ & $450$ & $19$ & $12$ & $0$ & $135$ & Sch\"{o}nheim bound & $136$\\
$40$ & $534$ & $13$ & $14$ & $204$ &  $192$ & \cite[Theorem 9]{Ho} & $193$\\
\end{tabular}
\caption{Some improved lower bounds on the size of $(v,\kappa,1)$-coverings}
\label{tab:improvedBoundExamples}
\end{table}

\section{Multigraphs with loops}\label{s:loops}

In this section we consider the variant of the problem in which we allow our multigraphs to have loops. Define a \emph{loop multigraph} to be a multigraph in which loops are allowed. We adopt the usual convention that a loop contributes 2 to the degree of a vertex in such a multigraph. In Theorem~\ref{t:loops} we show that it is not hard to determine the minimum value of $\alpha_k(G)$ over all loop multigraphs $G$ with a specified degree sequence. For the sake of concision, Theorem~\ref{t:loops} considers only positive degree sequences. Of course, adding a vertex of degree 0 to any loop multigraph increases its $k$-independence number by exactly 1.

A \emph{matching} is a (simple) 1-regular graph. A \emph{dominating set} in a loop multigraph $G$ is a subset $T$ of $V(G)$ such that each vertex in $V(G)\setminus T$ is adjacent in $G$ to a vertex of $T$. We will make use of the following well-known result on dominating sets (see \cite[p 41]{HayHedSla} for example).

\begin{lemma}\label{l:dom}
Let $G$ be a loop multigraph in which each vertex has a neighbour other than itself. Then $G$ has a dominating set of size at most $\lfloor\frac{1}{2}|V(G)|\rfloor$.
\end{lemma}

\begin{proof}
Let $T$ be a minimal dominating set for $G$. Then, because each vertex of $G$ has a neighbour other than itself and because $T$ is minimal, each vertex in $T$ is adjacent to at least one vertex in $V(G) \setminus T$. So $V(G) \setminus T$ is also a dominating set for $G$. One of $T$ or $V(G) \setminus T$ has size at most $\lfloor\frac{1}{2}|V(G)|\rfloor$.
\end{proof}

\begin{theorem}\label{t:loops}
Let $D$ be a degree sequence each of whose elements is positive, and let $k$ be a fixed positive integer. The minimum value of $\alpha_k(G)$ over all loop multigraphs $G$ with degree sequence $D$ is
\begin{itemize}
    \item
$|\{x \in D: x<k\}|$ if $k$ is even;
    \item
$\max\left(|\{x \in D: x<k\}|,\big\lceil\frac{1}{2}|\{x \in D: x \leq k\}|\big\rceil\right)$ if $k$ is odd.
\end{itemize}
\end{theorem}

\begin{proof}
Fix a degree sequence $D$ and let $\alpha_k(D)$ be the minimum value of $\alpha_k(G)$ over all loop multigraphs $G$ with degree sequence $D$. Let $s=|\{x \in D: x<k\}|$ and $c=|\{x \in D: x=k\}|$. Clearly $\alpha_k(D) \geq s$.

\noindent {\bf Case 1.} Suppose that $k$ is even. It suffices to exhibit a loop multigraph $G$ with degree sequence $D$ for which $\alpha_k(G) \leq s$. Take $G$ to be any loop multigraph with degree sequence $D$ in which each vertex of odd degree is incident with exactly one non-loop edge and each vertex of even degree is incident only with loop edges. In $G$, any vertex of degree at least $k$ is incident with at least $\frac{k}{2}$ loops and hence cannot be in any $k$-independent set. Thus $\alpha_k(G) \leq s$ as required.

\noindent {\bf Case 2.} Suppose that $k$ is odd. We first show that $\alpha_k(D) \geq \lceil\frac{1}{2}(s+c)\rceil$. Note that, because $k$ is odd, any vertex of degree $k$ in a loop multigraph has at least one neighbour other than itself. Let $G$ be an arbitrary loop multigraph with degree sequence $D$. Let $G'$ be the loop multigraph obtained from $G$ by deleting all vertices of degree greater than $k$. Let $G''$ be the loop multigraph obtained from $G'$ by deleting all vertices of $G'$ that are incident only with loops, noting that all such vertices have degree at most $k-1$. Using Lemma~\ref{l:dom}, take $T$ to be a dominating set for $G''$ of size at most $\lfloor\frac{1}{2}|V(G'')|\rfloor$. Then $V(G') \setminus T$ is a $k$-independent set in $G$ and $|V(G') \setminus T| \geq \lceil\frac{1}{2}(s+c)\rceil$ because $|V(G')|=s+c$ and $|T| \leq \lfloor\frac{1}{2}|V(G'')|\rfloor \leq \lfloor\frac{1}{2}(s+c)\rfloor$. So we do indeed have $\alpha_k(D) \geq \lceil\frac{1}{2}(s+c)\rceil$.

It remains to exhibit a loop multigraph $G$ with degree sequence $D$ for which $\alpha_k(G) \leq \max(s,\lceil\frac{1}{2}(s+c)\rceil)$. Let $D=\{d_1,\ldots,d_n\}$ where $d_1 \leq \ldots \leq d_n$, and note that $1 \leq d_i < k$ for each $i \in \{1,\ldots,s\}$, $d_i = k$ for each $i \in \{s+1,\ldots,s+c\}$, and $d_i > k$ for each $i \in \{s+c+1,\ldots,n\}$. Given two matchings $M_1$ and $M_2$ that we will select, we will define $G$ to be the unique loop multigraph on vertex set $V=\{1,\ldots,n\}$ so that the edges in $E(M_1) \cup E(M_2)$ are edges of $G$, all other edges of $G$ are loops, and $\deg_G(i)=d_i$ for each $i \in V$. Let $M_1$ be the matching such that
\[
E(M_1)=
\left\{
  \begin{array}{ll}
    \big\{\{i,s+i\}:i \in \{1,\ldots,c\}\big\} & \hbox{if $c \leq s$;} \\
    \big\{\{i,s+i\}:i \in \{1,\ldots,s\}\big\} \cup \big\{\{2s+2i-1,2s+2i\}:i \in \{1,\ldots,\lfloor\frac{c-s}{2}\rfloor\}\big\} & \hbox{if $c >s$.}
  \end{array}
\right.\]
Next select any matching $M_2$ such that
\[V(M_2)=\{i \in V(M_1):\hbox{$d_i$ is even}\} \cup \{i \in V \setminus V(M_1):\hbox{$d_i$ is odd}\},\]
noting that the set specified for $V(M_2)$ has even cardinality because $|\{i \in V(M_1):\hbox{$d_i$ is even}\}| \equiv |\{i \in V(M_1):\hbox{$d_i$ is odd}\}| \mod{2}$ since $|V(M_1)|$ is even. Now form the loop multigraph $G$ as described above. To see that $G$ exists note that, by the definition of $M_2$, $d_i$ is odd for each $i \in V$ that is in precisely one of $V(M_1)$ and $V(M_2)$ and $d_i$ is even (and hence at least 2) for each $i \in V$ that is in both or neither of $V(M_1)$ and $V(M_2)$.

Note that $V(M_1) \subseteq \{1,\ldots,s+c\}$. So, in $G$, any vertex in $\{s+c+1,\ldots,n\}$ is incident with at least $\frac{k+1}{2}$ loops and hence cannot be in any $k$-independent set. Thus a $k$-independent set in $G$ must be a subset of $\{1,\ldots,s+c\}$. Furthermore, it can be seen that any $k$-independent set in $G$ can contain at most one endpoint of each edge in $M_1$. Thus $\alpha_k(G) \leq s+c-|E(M_1)|$ and hence $\alpha_k(G) \leq \max(s,\lceil\frac{1}{2}(s+c)\rceil)$ because $|E(M_1)|=\min(c,\lfloor\frac{1}{2}(s+c)\rfloor)$.
\end{proof}

\section{Simple graphs}\label{s:simple}

We conclude by noting that, for the case of simple graphs, the most natural modification of the bound of Theorem~\ref{kIndepBound} fails. To see this, suppose $k=3$ and consider the degree sequence $D=\{1,3,4,4,4,5,5\}$. Under the natural modification of our definitions for $\Omega$ and $b$ we would have $\Omega(D)=\{0,3,3,3,3,4\}$, $\Omega^2(D)=\{0,2,2,2,2\}$ and hence $b(D)=5$. However, Figure~\ref{fig:simpleCE} shows a simple graph with degree sequence $D$ such that every $3$-independent set has size at most $4$.  Adapting the techniques of this paper to the case of simple graphs is an area of ongoing research.
\begin{figure}[h!]
\centering
 \includegraphics[scale=1]{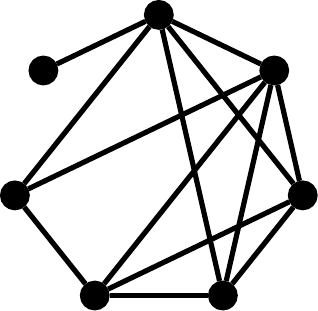}
\caption{Simple graph with degree sequence $\{1,3,4,4,4,5,5\}$ and no $3$-independent set of size 5}
\label{fig:simpleCE}
\end{figure}

\medskip
\noindent\textbf{Acknowledgments:} Thanks to an anonymous referee for pointing out an oversight in our treatment of multigraphs with loops. Thanks also to Rakhi Singh for preliminary computations concerning covering numbers. This work was supported by Australian Research Council grants DP150100506 and FT160100048.


  \let\oldthebibliography=\thebibliography
  \let\endoldthebibliography=\endthebibliography
  \renewenvironment{thebibliography}[1]{%
    \begin{oldthebibliography}{#1}%
      \setlength{\parskip}{0.4ex plus 0.1ex minus 0.1ex}%
      \setlength{\itemsep}{0.4ex plus 0.1ex minus 0.1ex}%
  }%
  {%
    \end{oldthebibliography}%
  }

\end{document}